\documentclass[12pt]{amsart}
\usepackage{geometry}                
\usepackage{graphicx}
\usepackage{amssymb}
\usepackage{amsthm}
\usepackage{epstopdf}
\newtheorem{theorem}{Theorem}
\newtheorem{lemma}{Lemma}
\newtheorem{proposition}{Proposition}

\numberwithin{equation}{section}

\title{An application of functional analysis to the Riemann zeta function}
\author{Kevin Smith}

\begin{document}
\maketitle

\begin{abstract} Lindel\"of conjectured that the Riemann zeta function $\zeta(\sigma+it)$ grows more slowly than any fixed positive power of $t$ as $t\rightarrow\infty$ when $\sigma\geq 1/2$. Hardy and Littlewood showed that this is equivalent to the existence of the $2k$th moments for all fixed $k\in\mathbb{N}$ and $\sigma>1/2$. In this paper we show that the completeness of Hilbert space implies that if the $2k$th moment exists for $\sigma>\sigma_k>1/2$ then it also exists on the line $\sigma=\sigma_k$.
\end{abstract}

\section{Introduction}
For fixed $\sigma\geq 1/2$ and $k\in\mathbb{N}$ we define the $2k$th moment of the Riemann zeta function by
\begin{eqnarray}\label{moment}
M_k(\sigma,T)=\int_1^T|\zeta(\sigma+it)|^{2k}dt\hspace{1cm}(T>1)\nonumber
\end{eqnarray}
and the number of ways of writing the natural number $n$ as a product of $k$ factors is denoted by $d_k(n)$. In 1923, Hardy and Littlewood \cite{HL2} showed that the Lindel\"of hypothesis is equivalent to the statement that
\begin{eqnarray}\label{asy0}
\lim_{T\rightarrow\infty}\frac{M_k(\sigma,T)}{T} =\sum_{n=1}^{\infty}\frac{d^2_k(n)}{n^{2\sigma}} \hspace{1cm} (\sigma>1/2,k\in\mathbb{N}).
\end{eqnarray}
 The cases $k=1$ and $2$ are classical results of Hardy and Littlewood \cite{HL} and Ingham \cite{I}, respectively, yet the asymptotic formula (\ref{asy0}) is not known to hold throughout the region $\sigma>1/2$ for any other value of $k$. 

The moments are convex functions of $\sigma$, a consequence of which is that there exists a non-decreasing sequence $1/2\leq \sigma_k< 1$ such that 
\begin{eqnarray}\label{asy3}
\lim_{T\rightarrow\infty}\frac{M_k(\sigma,T)}{T} =\sum_{n=1}^{\infty}\frac{d^2_k(n)}{n^{2\sigma}} \hspace{1cm} (\sigma>\sigma_k)
\end{eqnarray}
while
\begin{eqnarray}\label{asy1/2}
 \limsup_{T\rightarrow\infty}\frac{M_k(\sigma,T)}{T}=\infty  \hspace{1cm}(1/2\leq \sigma< \sigma_k \textrm{ if }\sigma_k>1/2)
\end{eqnarray}
and
\begin{eqnarray}\label{mbound}
M_k(\sigma_k,T)\ll T^{1+\epsilon} \hspace{1cm}(\epsilon>0)
\end{eqnarray} 
in general (Titchmarsh \cite{Titch}, Section 7.9). Upper bounds for $\sigma_k$ and estimates for the rate of convergence in (\ref{asy3}) were given by Iv\'ic (\cite{IV}, Section 8.5) yet, beyond the theorems of Hardy and Littlewood and Ingham cited above, which give  $\sigma_1=\sigma_2=1/2$, the strongest known estimates are $\sigma_3\leq 7/12$ due to Iv\'ic cited above and $\sigma_4\leq 5/8$, the latter being a consequence of Heath-Brown's  twelfth moment estimate \cite{HB}. \\

The question\footnote{This question was brought to my attention by Peter Sarnak during the \emph{50 years of number theory and random matrix theory} conference at the IAS (2022).} of the behaviour on a hypothetical line $\sigma_k>1/2$, i.e. whether (\ref{mbound}) can be improved if $\sigma_k>1/2$, is more difficult. The purpose of this paper is to show that the completeness of Hilbert space implies the following. 
\begin{theorem}\label{main}
If $\sigma_k>1/2$ then
 \begin{eqnarray}\label{asy4}
\lim_{T\rightarrow\infty}\frac{M_k(\sigma_k,T)}{T} =\sum_{n=1}^{\infty}\frac{d^2_k(n)}{n^{2\sigma_k}}.
\end{eqnarray}
\end{theorem}

The result is a consequence of the two lemmas below. Let $B^2=B^2[1,\infty]$ denote the space of Besicovitch almost-periodic functions on $[1,\infty]$, which we discuss in Section \ref{beslemmas}. We show in Section \ref{lemma1proof} that

\begin{lemma}\label{hlemma}
If $\sigma>\sigma_k$ then $\zeta^k(\sigma+it)\in B^2$.
\end{lemma}

In Section \ref{lemma2proof} we see that if $\sigma(n)>\sigma_k>1/2$ is a decreasing sequence with limit $\sigma_k$ then the sequence of functions $f^k_n(t)=\zeta^k(\sigma(n)+it)$ is a Cauchy sequence in $B^2$ which therefore has a limit, say $f^k_{\infty}\in B^2$. From this we deduce that the pointwise limit $\zeta^k(\sigma_k+it)$ also belongs to $B^2$ via classical  Abelian and Tauberian arguments. That is

\begin{lemma}\label{clemma}
If $\sigma_k>1/2$ then $\zeta^k(\sigma_k+it)\in B^2$. 
\end{lemma}

Evidently, Theorem \ref{main} amounts to the statement that  \emph{if $\sigma_k>1/2$ then $\zeta^k(\sigma+it)\in B^2$ if and only if $\sigma\geq \sigma_k$}, which may be compared with the statement that \emph{if $\sigma_k=1/2$ then $\zeta^k(\sigma+it)\in B^2$ if and only if $\sigma> 1/2$}.

\section{Besicovitch almost-periodic functions. Concentration on a null set}\label{beslemmas}
Let $f,g\in L^2_{\textrm{loc}}[1,\infty]$. For measurable subsets $S\subseteq [1,\infty]$ we define the natural density
\begin{eqnarray}\label{densitydef}
|S|=\lim_{T\rightarrow\infty}\frac{\textrm{meas}\left(S\cap [1,T]   \right)}{T}\nonumber
\end{eqnarray}
and if $|S|=0$ we say that $S$ is null. We also define 
$\textrm{Pr}\left(|g|\geq x\right)=|X|$ where $X\subseteq [1,\infty]$ is the subset on which $|g|\geq x$. 

Our first proposition is an improvement on the trivial bound 
\begin{eqnarray}\label{holderbound}
\frac{\left|\int_{1}^Tf^2gdt\right|}{\int_{1}^T|f|^2dt}\leq\textrm{ess sup}_{[1,T]}|g|\nonumber
\end{eqnarray}
when $T$ is large if $f\notin L^2[1,\infty]$ and $f$ is not concentrated on a null set, by which we mean the following. 

\begin{proposition}\label{CNS} The following statements are equivalent. If $f$ is as such, we say that $f$ is not concentrated on a null set. Otherwise, $f$ is concentrated on a null set. 
\begin{itemize}
\item[(\emph{a})]{If $S$ is null then 
\begin{eqnarray}\label{meas}
\int_{S\cap[1,T]}|f|^2dt=o\left(\int_{1}^T|f|^2dt\right)\hspace{1cm}(T\rightarrow\infty).
\end{eqnarray}
}
\item[(\emph{b})]{If $g$ is bounded then
\begin{eqnarray}\label{prob}
\limsup_{T\rightarrow\infty}\frac{\left|\int_{1}^Tf^2gdt\right|}{\int_{1}^T|f|^2dt}\leq \sup\{x:\Pr(|g|\geq x)>0\}.
\end{eqnarray}
}

\end{itemize}
Moreover, if $\int_{1}^T|f|^2dt\gg T$ we may include 
\begin{itemize}
\item[(\emph{c})]{If $S$ is null then for every $\epsilon>0$ there is an $N=N(\epsilon)$ and a bounded function $f_{N}$  such that 
\begin{eqnarray}\label{ot}
\limsup_{T\rightarrow\infty}\frac{\int_{S\cap[1,T]}|f-f_{N}|^2dt}{\int_{1}^T|f|^2dt}< \epsilon.
\end{eqnarray}
}
\end{itemize}
If also $\int_{1}^T|f|^2dt\ll T$ then condition \emph{(}c\emph{)} may be replaced with 
\begin{itemize}
\item[(\emph{d})]{For every $\epsilon>0$ there is an $N=N(\epsilon)$ and a bounded function $f_{N}$ such that 
\begin{eqnarray}\label{sup}
\limsup_{T\rightarrow\infty}\frac{1}{T}\int_{1}^{T}|f-f_{N}|^2dt< \epsilon.
\end{eqnarray}
}
\end{itemize}
\end{proposition}

Proposition \ref{CNS} is proved in Section \ref{prop1}. \\

We now write
\begin{eqnarray} \label{ip}
\langle f, g\rangle=\lim_{T\rightarrow\infty}\frac{1}{T}\int_{1}^{T}f(t)\overline{g(t)}dt
\end{eqnarray}
and $\|f\|=\sqrt{\langle f,f\rangle}$ when the limits exists, and note that if (\ref{sup}) holds then $\|f\|=\lim_{N\rightarrow \infty}\|f_{N}\|$, by the triangle inequality. For example, this is the case when the $f_{N}$ are Bohr's uniformly almost-periodic functions. These are bounded, continuous and represented almost everywhere by a generalised Fourier series
\begin{eqnarray}\label{Bohr}
g= \sum_{\lambda \in \mathbb{R}} c_{\lambda} e_{\lambda}\nonumber
\end{eqnarray}
in which the Fourier coefficients $c_{\lambda}= \langle g,e_{\lambda}\rangle$ are square-summable and $\|f_N\|^2=\sum_{\lambda\in\mathbb{R}}|\langle f_N,e_{\lambda}\rangle|^2$. In this case, the $f$ in (\ref{sup}) belongs to the closure of the space of Bohr functions in the seminorm $\|\cdot\|$. That is, for trivial $\phi$ (i.e. $\|\phi\|=0$) the equivalence class $\phi+f$ belongs to $B^2$---Besicovitch \cite{Bes} proved that $B^2$ is a Hilbert space, i.e. that the analogue of the Riesz-Fischer theorem holds for this inner product space. 

For $f,g\in B^2$ the inner product  (\ref{ip}) exists, the coefficients $\langle f,e_{\lambda}\rangle$ are necessarily non-zero for at most countably many $\lambda\in\mathbb{R}$ and the Parseval relation
\begin{eqnarray}\label{parseval}
\langle f,g\rangle=\sum_{\lambda\in\mathbb{R}}\langle f,e_{\lambda}\rangle \overline{\langle g,e_{\lambda}\rangle}
\end{eqnarray}
holds. Conversely, if $f\in B^2$ then for instance we may take $f_N$ to be a partial sum
\begin{eqnarray}\label{partial}
f_{N}= \sum_{n\leq N} \langle f,e_{\lambda_n}\rangle e_{\lambda_n}
\end{eqnarray}
and verify (\ref{sup}) for this sequence using Bessel's inequality.

In particular, Proposition \ref{CNS} shows that 

\begin{lemma}\label{CNSlemma}
If $f\in B^2$ is non-trivial then $f$ is not concentrated on a null set. 
\end{lemma}

We can also prove the following. 
\begin{lemma}\label{kncnslemma}Let $k\in\mathbb{N}$ be fixed. If $f\in B^2$ is non-trivial and $f^k$ is locally square integrable then $f^k\in B^2$ if and only if $f^k$ is not concentrated on a null set.
\end{lemma}
\begin{proof}
The ``only if'' statement follows from Lemma \ref{CNSlemma}, so we prove the converse. Assuming that $f\in B^2$, for every $\epsilon>0$ there is an $N(\epsilon)$ and a sequence of Bohr functions $f_N$ such that $\|f-f_N\|<\epsilon$ for $N>N(\epsilon)$, so there is a sequence of null sets $S_1(N)$ such that 
\begin{eqnarray}\label{pointwise1}
\left|f(t)-f_N(t)   \right|<\epsilon \hspace{1cm}(t\in S_1^c(N)).
\end{eqnarray}
Moreover, since $f_N$ converges in $B^2$ so $\|f_N\|\rightarrow\|f\|$ as $N\rightarrow\infty$, there is a constant $C$ and a sequence of null sets $S_2(N)$ such that 
\begin{eqnarray}\label{pointwise3}
\left|f_N(t)   \right|<C \hspace{1cm}(t\in S_2^c(N))
\end{eqnarray}
uniformly in $N$.

We set 
\begin{eqnarray}\label{setdef}
S(N)=S_1(N)\cup S_2(N)
\end{eqnarray}
and use the binomial expansion to see that 
\begin{eqnarray}\label{binom}
\left|f^k(t)-f^k_N(t)\right|&<&\sum_{j=1}^k{k\choose j}C^{k-j}\epsilon^j
\nonumber\\
&=&(C+\epsilon)^k-C^k\ll_k\epsilon\hspace{1cm}\left(t\in S^c(N)\right),
\end{eqnarray}
by (\ref{pointwise1}) and (\ref{pointwise3}). Assuming that $f^k$ is locally square-integrable, by (\ref{binom}) and the triangle inequality we have
\begin{eqnarray}\label{meanconvergence}
&&\limsup_{T\rightarrow\infty}\left|\left(\frac{1}{T}\int_{S^c(N)\cap [1,T]}\left|f(t)\right|^{2k}dt\right)^{1/{2}}-\|f^k_N\| \right|
\nonumber\\
&=&\limsup_{T\rightarrow\infty}\left|\left(\frac{1}{T}\int_{S^c(N)\cap [1,T]}\left|f(t)\right|^{2k}dt\right)^{1/{2}}-\left(\frac{1}{T}\int_{S^c(N)\cap [1,T]}\left|f_N(t)\right|^{2k}dt\right)^{1/{2}} \right|\nonumber\\
&\leq&\limsup_{T\rightarrow\infty}\left(\frac{1}{T}\int_{S^c(N)\cap [1,T]}\left|f^k(t)-f^k_N(t)\right|^{2}dt\right)^{1/{2}}\ll_k\epsilon\hspace{1cm}(N>N(\epsilon)),\nonumber\\
\end{eqnarray}
which shows that
\begin{eqnarray}\label{complement}
\frac{1}{T}\int_{S^c(N)\cap [1,T]}\left|f(t)\right|^{2k}dt\sim \|f_N^k\|^2\hspace{1cm}(N,T\rightarrow\infty).
\end{eqnarray}
By  (\ref{complement}) we have 
\begin{eqnarray}
\frac{1}{T}\int_{S(N)\cap [1,T]}\left|f(t)\right|^{2k}dt \ll_k \frac{\int_{S(N)\cap [1,T]}\left|f(t)\right|^{2k}dt}{\int_{S^c(N)\cap [1,T]}\left|f(t)\right|^{2k}dt}\nonumber
\end{eqnarray}
and, since we are assuming $f^k$ is not concentrated on a null set, this is 
\begin{eqnarray}
\ll_k \frac{\int_{S(N)\cap [1,T]}\left|f(t)\right|^{2k}dt}{\int_1^T\left|f(t)\right|^{2k}dt}\ll_k \epsilon\hspace{1cm}(T>T(\epsilon)),\nonumber
\end{eqnarray}
so for fixed $k\in\mathbb{N}$ we have
\begin{eqnarray}\label{obound}
\lim_{T\rightarrow\infty}\frac{1}{T}\int_{S(N)\cap [1,T]}\left|f(t)\right|^{2k}dt=0.
\end{eqnarray}
Using (\ref{binom}) again, (\ref{obound}) and Cauchy-Schwarz, it follows that
\begin{eqnarray}
\limsup_{T\rightarrow\infty}\frac{1}{T}\int_1^T\left |f^k(t)-f^k_N(t)\right|^2dt &\ll_k& \epsilon^2+\limsup_{T\rightarrow\infty}\frac{1}{T}\int_{S(N)\cap [1,T]}\left|f^k(t)-f^k_N(t)\right|^{2}dt\nonumber\\
&=&\epsilon^2\nonumber
\end{eqnarray}
so $f^k_N\rightarrow f^k$ in the $B^2$ metric. 
\end{proof}

\section{Proof of Lemma \ref{hlemma}}\label{lemma1proof}
We assume $1/2<\sigma<1$ and write $\zeta=\zeta(\sigma+it)$ unless otherwise specified.  We begin by proving that $\zeta\in B^2$ when  $1/2<\sigma<1$. Fix $\delta>0$ and denote by $R$ the rectangle with vertices $[\sigma+iT,1+\delta+iT,1+\delta+i,\sigma+i]$. Consider the integral 
\begin{eqnarray}\label{inter}
\frac{1}{2iT}\int_{R}\zeta^{k}(s)e^{\lambda(2\sigma-s)}ds.
\end{eqnarray}
Since we have the subconvexity bound 
\begin{eqnarray}\label{subbound}\zeta(\sigma+it)\ll t^{1/6-A}\hspace{1cm}(\sigma\geq 1/2)
\end{eqnarray} for some $A>0$, for $k\leq 6$ the sum of the horizontal segments is
\begin{eqnarray}\label{ppo}
\ll \frac{1}{T} \int_{\sigma}^{1+\delta}|\zeta(r+iT)|^{k}e^{\lambda(2\sigma-r)}dr+\frac{e^{\lambda \sigma}}{T}\ll \frac{e^{\lambda \sigma}}{T^{A}}
\end{eqnarray}
and the sum of the vertical segments is
\begin{eqnarray}\label{vgg}
&&\frac{e^{2\lambda\sigma}}{T}\int_{1}^{T}\zeta^{k}(1+\delta+it)e^{-\lambda(1+\delta+it)}dt-\frac{e^{\lambda\sigma}}{T}\int_{1}^{T}\zeta^{k}(\sigma+it)e^{-i\lambda t}dt\nonumber\\
&=&\sum_{n=1}^{\infty} \frac{d_{k}(n)}{n^{1+\delta}}\frac{e^{\lambda(2\sigma-1-\delta)}}{T}\int_{1}^{T}e^{-i(\lambda+\log n) t}dt -\frac{e^{\lambda\sigma}}{T}\int_{1}^{T}\zeta^{k}(\sigma+it)e^{-i\lambda t}dt
\end{eqnarray}
because the Dirichlet series above is absolutely convergent. By (\ref{inter}), (\ref{ppo}) and (\ref{vgg}) we see that 
\begin{eqnarray}
\frac{1}{T}\int_{1}^{T}\zeta^{k}(\sigma+it)e^{-i\lambda t}dt&=&e^{\lambda(\sigma-1-\delta)}\sum_{n=1}^{\infty} \frac{d_{k}(n)}{n^{1+\delta}}\frac{1}{T}\int_{1}^{T}e^{-i(\lambda+\log n) t}dt+O(T^{-A})\nonumber\\
&=&\left\{
                \begin{array}{ll}
                  \frac{d_{k}(n)}{n^{\sigma}} +o\left(1\right)\hspace{1.25cm}(\lambda=-\log n)  \\
               o\left(1\right)\hspace{2.55cm}(\textrm{otherwise}),
                \end{array}
              \right.\nonumber
\end{eqnarray}
that is 
\begin{eqnarray}\label{coefficients}
\langle \zeta^k,e_{\lambda}\rangle = \left\{
                \begin{array}{ll}
                  \frac{d_{k}(n)}{n^{\sigma}} \hspace{1.25cm}(\lambda=-\log n \hspace{0.3cm}(n\in\mathbb{N}))  \\
               0\hspace{1.83cm}(\textrm{otherwise}).
                \end{array}
              \right.
              \end{eqnarray}
We now set  
\begin{eqnarray}
\zeta_{k,N}=\sum_{n\leq N}\frac{d_{k}(n)}{n^{\sigma+i\cdot}}.\nonumber
\end{eqnarray}
Then, for $\sigma>\sigma_k$ and $k\leq 6$, we use (\ref{asy3}) and (\ref{coefficients}) to compute
\begin{eqnarray}\label{beseq}
\|\zeta^k-\zeta_{k,N}\|^2 &=& \|\zeta^k\|^2+\|\zeta_{k,N}\|^2-2\Re\langle \zeta^k,\overline \zeta_{k,N}    \rangle \nonumber\\
&=&\|\zeta^k\|^2-\|\zeta_{k,N}\|^2\nonumber\\
&= &\sum_{n> N} \frac{d^2_{k}(n)}{n^{2\sigma}}<\epsilon\hspace{1cm}(N>N(\epsilon)).\nonumber
\end{eqnarray}
In particular, the above and the fact that $\sigma_1=1/2$ together show that $\zeta\in B^2$ when $1/2<\sigma<1$.  Since $\zeta^k$ is locally square integrable, it follows from Lemma \ref{kncnslemma} that $\zeta^k\in B^2$ if and only if $\zeta^k$ is not concentrated on a null set. To see that this is the case when $\sigma>\sigma_k$, we note that if  $1/2<\sigma<1$ then
\begin{eqnarray}\label{cds}
\frac{1}{T}\int_1^T|\zeta(\sigma+it)|^{2k}dt\sim \sum_{r=1}^{\infty} \frac{d^2_k(r)}{r^{2\sigma}}+\frac{1}{T}\int_{S(N)\cap[1,T]}|\zeta(\sigma+it)|^{2k}dt
\end{eqnarray}
as $N,T\rightarrow\infty$, by (\ref{complement}). Since the left hand side of (\ref{cds}) is $\sim \sum_{r=1}^{\infty} d^2_k(r)r^{-2\sigma}$ when  $\sigma>\sigma_k$  by (\ref{asy3}), while $\int_1^T|\zeta(\sigma+it)|^{2k}dt\gg T$ and $|\zeta(\sigma+it)|^{2k}$ is bounded on the complements $S^c(N)$ for fixed $k$ by (\ref{binom}), we see that $\zeta^k(\sigma+it)$ is not concentrated on a null set when $\sigma>\sigma_k$.

\section{Proof of Lemma \ref{clemma}}\label{lemma2proof}
Let $\sigma_k<\sigma(n)<1$ be a decreasing sequence with limit $\sigma_k>1/2$ and put 
\begin{eqnarray}
f_n(t)=\zeta (\sigma(n)+it).\nonumber 
\end{eqnarray}
Since $f_n^k\in B^2$ by Lemma \ref{hlemma}, we suppose that $m>n$ and using (\ref{parseval}) compute 
\begin{eqnarray}
\|f^k_m-f^k_n\|^2 &=&\sum_{r=1}^{\infty}\frac{d^2_k(r)}{r^{2\sigma(m)}}+\sum_{r=1}^{\infty}\frac{d^2_k(r)}{r^{2\sigma(n)}}-2\sum_{r=1}^{\infty}\frac{d^2_k(r)}{r^{\sigma(m)+\sigma(n)}}\nonumber\\
&=& \sum_{r=2}^{\infty}\frac{d^2_k(r)}{r^{2\sigma(m)}}\left(1-r^{\sigma(m)-\sigma(n)}\right)^2
\nonumber\\
&<& \sum_{r=2}^{\infty}\frac{d^2_k(r)}{r^{2\sigma_k}}\left(1-r^{\sigma_k-\sigma(n)}\right)
\nonumber\\
&<& \left(1-M^{\sigma_k-\sigma(n)}\right)\sum_{2\leq r\leq M}^{}\frac{d^2_k(r)}{r^{2\sigma_k}}+\sum_{r> M}^{}\frac{d^2_k(r)}{r^{2\sigma_k}}
\nonumber\\
&\ll_{k,\sigma_k}& 1-M^{\sigma_k-\sigma(n)}+M^{1-2\sigma_k+\epsilon}\nonumber
\end{eqnarray}
for any fixed $\epsilon>0$. Thus, for any $\delta>0$, we may take $M(\delta)=\delta^{(1-2\sigma_k+\epsilon)^{-1}}$ and choose $N=N(\delta)$ such that 
\begin{eqnarray}\label{uniform}
\sigma(N)&< &\sigma_k+\frac{\log\left(\frac{1}{1-\delta}\right)}{\log M}\nonumber\\
&=&\sigma_k+(2\sigma_k-1-\epsilon)\frac{\log\left(\frac{1}{1-\delta}\right)}{\log\frac{1}{\delta}},\nonumber
\end{eqnarray}
in which case
\begin{eqnarray}\label{Cauchys}
\|f^k_m-f^k_n\|^2<2\delta\hspace{1cm}(m>n>N)\nonumber
\end{eqnarray}
so the sequence $f^k_n$ is a Cauchy sequence in $B^2$. As such, there is an $f^k_{\infty}\in B^2$ such that $f^k_n\rightarrow f^k_{\infty}$ in $B^2$, and of course 
\begin{eqnarray}\label{norm}
\|f^k_{\infty}\|^2=\lim_{n\rightarrow\infty}\|f^k_n\|^2=\sum_{r=1}^{\infty} \frac{d^2_k(r)}{r^{2\sigma_k}}.
\end{eqnarray}

Next, let $k\in\mathbb{N}$ be fixed and $z=x+iy$. We consider the Laplace transform 
\begin{eqnarray}\label{Fdef}
L_{\sigma,k}(z)=\int_1^{\infty}|\zeta(\sigma+it)|^{2k}e^{-zt}dt\hspace{1cm}(x>0)
\end{eqnarray}
and note that $L_{\sigma,k}(z)$ is analytic in this region because $\zeta(\sigma+it)$ is polynomially bounded. Moreover, if $x>0$ then
 \begin{eqnarray}
f^k_n(t)e^{-xt}\rightarrow \zeta^k(\sigma_k+it)e^{-xt}\hspace{1cm}(t\in[1,\infty])\nonumber
\end{eqnarray} uniformly as $n\rightarrow\infty$. As such
\begin{eqnarray}\label{Fcon}
L_{\sigma_k,k}(z)
&=&\int_1^{\infty}\lim_{n\rightarrow\infty}|f_n(t)|^{2k}e^{-zt}dt\nonumber\\
&=&\lim_{n\rightarrow\infty}L_{\sigma(n),k}(z)
\hspace{3cm}(x>0).
\end{eqnarray}

Now, Abel's theorem for integrals ensures that if $f,g\in B^2$ then
\begin{eqnarray}\label{abel}
\langle f,g\rangle=\lim_{x\searrow 0}x\int_{1}^{\infty}f(t)\overline{g(t)}e^{-xt}dt
\end{eqnarray}
so, by (\ref{uniform}), (\ref{Fcon}) and (\ref{abel}), we have
\begin{eqnarray}
\|f^k_{\infty}\|^2&=&\lim_{n\rightarrow\infty}\|f^k_{n}\|^2\nonumber\\
&=&\lim_{n\rightarrow\infty} \lim_{x\searrow 1}
xL_{\sigma(n),k}(x)\nonumber\\
&=&\lim_{x\searrow 1}\lim_{n\rightarrow\infty} 
xL_{\sigma(n),k}(x)\nonumber\\
&=&\lim_{x\searrow 1}
xL_{\sigma_k,k}(x),\nonumber
\end{eqnarray}
where interchanging the limits in the third line above is justified because the $n$ limit is achieved pointwise in $x$ and the $x$ limit is achieved uniformly in $n$ because $f_n^k$ is a Cauchy sequence in $B^2$. 

Thus if $k\in\mathbb{N}$ and $\sigma_k>1/2$ are fixed then 
\begin{eqnarray}\label{asy4}
\int_1^{\infty}|\zeta(\sigma_k+it)|^{2k}e^{-xt}dt\sim \frac{1}{x}\sum_{r=1}^{\infty} \frac{d^2_k(r)}{r^{2\sigma_k}}\hspace{1cm}(x\searrow 0)
\end{eqnarray}
and, by the version of the Hardy-Littlewood Tauberian theorem (Titchmarsh \cite{Titch}, Section 7.2), (\ref{asy4}) implies that
\begin{eqnarray}\label{cds1}
\frac{1}{T}\int_1^T|\zeta(\sigma_k+it)|^{2k}dt\sim \sum_{r=1}^{\infty} \frac{d^2_k(r)}{r^{2\sigma_k}}\hspace{1cm}(T\rightarrow \infty).
\end{eqnarray}
Lastly, it is easy to see that (\ref{cds1}) implies Lemma \ref{clemma} by taking $\sigma=\sigma_k$ in (\ref{cds}) and noting that consequently $\zeta^k(\sigma_k+it)$ is not concentrated on a null set.

\section{Proof of Proposition \ref{CNS}}\label{prop1}
Let $g$ be bounded, assume (\ref{meas}) and denote by $S^c$ the complement of $S$ in $[1,\infty]$. Then for every null set $S$ we have 
\begin{eqnarray}\label{prob3}
\left|\int_{1}^Tf^2gdt \right|&=&\left|\int_{S^c\cap [1,T]}f^2gdt+\int_{S\cap [1,T]}f^2gdt\right|\nonumber\\
&\leq&\left|\int_{S^c\cap [1,T]}f^2gdt\right|+o\left(\int_{1}^T|f|^2dt\right)\nonumber\\
&\leq&(1+o(1))\left(\int_{1}^T|f|^2dt\right)\sup_{t\in S^c\cap [1,T]}|g(t)|.\nonumber
\end{eqnarray}
Since $S$ is an arbitrary null set, we minimise over those $S$ giving 
\begin{eqnarray}\label{prob4}
\limsup_{T\rightarrow\infty}\frac{\left|\int_{1}^Tf^2gdt\right|}{\int_{1}^T|f|^2dt}\leq \inf_{|S|=0}
\sup_{t\in S^c}|g(t)|= \sup\{x:\Pr(|g|\geq x)>0\}
\end{eqnarray}
which gives (\ref{prob}). Conversely, assume (\ref{prob}) and take $g=e^{-2i\theta_f}\chi_{S}$ where $\chi_S$ is the characteristic function of $S$. Then 
\begin{eqnarray}
\limsup_{T\rightarrow\infty}\frac{\int_{S\cap [1,T]}|f|^2dt}{\int_{1}^T|f|^2dt}=\limsup_{T\rightarrow\infty}\frac{\left|\int_{1}^Tf^2gdt\right|}{\int_{1}^T|f|^2dt}\leq \sup\{x:\Pr(|g|\geq x)>0\}=0\nonumber
\end{eqnarray}
which gives (\ref{meas}).

Now suppose that $\int_{1}^T|f|^2dt\gg T$. We have 
\begin{eqnarray}\label{prob5}
\left(\frac{\int_{S\cap [1,T]}|f|^2dt}{\int_{1}^T|f|^2dt}\right)^{1/2}
&\sim& \left| \left(\frac{\int_{S\cap [1,T]}|f|^2dt}{\int_{1}^T|f|^2dt}\right)^{1/2} - \left(\frac{\int_{S\cap [1,T]}|f_{N,S}|^2dt}{\int_{1}^T|f|^2dt}\right)^{1/2} \right|\nonumber\\
&\leq&  \left(\frac{\int_{S\cap [1,T]}|f-f_{N,S}|^2dt}{\int_{1}^T|f|^2dt}\right)^{1/2} \nonumber
\end{eqnarray}
so (\ref{ot}) implies (\ref{meas}). Conversely, (\ref{meas}) implies (\ref{ot}) by taking $f_{N,S}=0$. 

To complete the proof, we note that if $\int_{1}^T|f|^2dt\ll T$ then $f$ is bounded except possibly on a null set $S=S(f)$, so there is already a bounded function $f_{N,S^c}$ such that $|f-f_{N,S^c}|^2<\epsilon$ on $S^c$. Thus, taking 
\begin{eqnarray}
f_N=f_{N,S}\chi_{S}+f_{N,S^c}\chi_{S^c},\nonumber
\end{eqnarray}
we have
\begin{eqnarray}
\frac{1}{T}\int_{1}^T|f-f_N|^2 dt &=&\frac{1}{T}\int_{S\cap [1,T]}|f-f_{N,S}|^2dt+\frac{1}{T}\int_{S^c\cap [1,T]}|f-f_{N,S^c}|^2dt\nonumber\\
&\ll & \frac{\int_{S\cap [1,T]}|f-f_{N,S}|^2dt}{\int_{1}^T|f|^2dt}+\frac{\epsilon}{T}\textrm{meas}\left(S^c\cap [1,T]   \right)
\nonumber
\end{eqnarray}
and using (\ref{ot}) gives (\ref{sup}). Conversely,  (\ref{sup}) implies  (\ref{ot}) if $\int_{1}^T|f|^2dt\gg T$.

\noindent\emph{Email address}:\texttt{ks614@exeter.ac.uk}

\end{document}